\documentclass[12pt]{amsart}  
\usepackage{amssymb} 
\usepackage[utf8]{inputenc}
\usepackage[english]{babel}
\usepackage{lineno,hyperref}
\usepackage{amsmath}
\usepackage{amscd}
\usepackage{amsthm}
\usepackage{amsfonts}
\usepackage{color}
\usepackage{xfrac}
\usepackage{mathtools}
\usepackage{graphicx}
\usepackage{diagrams} 
\usepackage{tikz}
\usetikzlibrary{arrows}
\usepackage{bbold}
\usepackage[T1]{fontenc}
\newtheorem{theorem}{Theorem}[section]
\newtheorem{proposition}[theorem]{Proposition}

\newtheorem{remark}[theorem]{Remark}
\newtheorem{question}[theorem]{Question}

\newtheorem{definition}[theorem]{Definition}

\newtheorem{supplement}[theorem]{Supplement}

\numberwithin{equation}{section}

\AtBeginDocument{\setlength{\linenumbersep}{\dimexpr\oddsidemargin+0.2in-\linenumberwidth\relax}} 

\begin{document}  
\title[Limits of manifolds in the Gromov-Hausdorff metric space]
{Limits of manifolds in the Gromov-Hausdorff metric space} 
\author[F. Hegenbarth and D.D. Repov\v{s}]{ 
Friedrich Hegenbarth and Du\v{s}an D. Repov\v{s}}
\address{Dipartimento di Matematica "Federigo Enriques", 
Universit\` a degli studi di Milano, 
20133 Milano, Italy}
\email{friedrich.hegenbarth@unimi.it} 
\address{Faculty of Education and Faculty of Mathematics and Physics, University of Ljubljana \&
Institute of Mathematics, Physics and Mechanics, 
1000 Ljubljana, Slovenia}
\email{dusan.repovs@guest.arnes.si}
\begin{abstract}  
We apply the Gromov-Hausdorff metric $d_G$ for  characterization of certain generalized manifolds.
Previously, we have proved that with respect to the metric $d_G,$ generalized $n$-manifolds are limits of spaces which are
obtained by gluing two topological $n$-manifolds by a controlled homotopy equivalence (the so-called $2$-patch spaces).
In the present paper, we consider the so-called {\sl manifold-like} generalized $n$-manifolds $X^{n},$ introduced in 1966
 by Marde\v{s}i\'{c}
and
Segal, which are characterized by
the
 existence of $\delta$-mappings $f_{\delta}$
 of $X^n$
onto closed manifolds $M^{n}_{\delta},$ for arbitrary small $\delta>0$, i.e.  
there exist
 onto maps $f_{\delta}\colon X^{n}\to M^{n}_{\delta}$
such that for every $u\in M^{n}_{\delta}$, $f^{-1}_{\delta}(u)$ has diameter less than $\delta$. 
We prove that  with respect to
the   metric $d_G,$ manifold-like generalized $n$-manifolds $X^{n}$
are limits of topological $n$-manifolds $M^{n}_{i}$.
Moreover, if  topological $n$-manifolds  $M^{n}_{i}$  satisfy a certain local contractibility condition $\mathcal{M}(\varrho, n)$, we prove that
generalized $n$-manifold  $X^{n}$ is resolvable.
\end{abstract}
\maketitle
\section{Introduction}\label{s1} 
This paper is a continuation of our systematic  study of the characterization problem for generalized $n$-manifolds,
$n\ge 5,$
(see 
Cavicchioli {\it et al.}~\cite{CHR, CHR16}
  and
Hegenbarth and Repov\v{s}~\cite{HeRe06,HeRe14,HeRe18,HeRe19,HeRe20,HeRe21}).
This is a very important class of spaces which in the algebraic sense
 strongly resemble topological manifolds, whereas in the geometric sense they can fail to be locally Euclidean at any point
(see e.g.,
Cannon~\cite{C},
Edwards~\cite{ Edw1978}, and
 Repov\v{s}~\cite{W,B,H}).
\begin{definition}

A  \textit{generalized $n$-manifold} $X^{n}$ is an $n$-dimensional metric absolute neighborhood retract 
(ANR) $X^{n}$ with
local homology 
\[
H_{*}(X^{n},X^{n}\setminus\{x\};\mathbb Z)\cong
H_{*}(\mathbb  R^{n},\mathbb R^{n}\setminus\{0\};\mathbb Z), \  \mbox{for every} \  x\in X.
\]
\end{definition}
We shall only consider {\it oriented} generalized $n$-manifolds {\it without boundary}
(i.e. $H_{n}(X^{n},X^{n}\setminus\{x\};\mathbb Z)\cong\mathbb Z,$ for every $x \in X^{n}$).
 Throughout the paper, we shall be assuming that $n\geq 5.$
 
\begin{definition}
Given  any $\delta >0,$ a  continuous map
 $f_{\delta}\colon X\to Y$ of a metric space $X$ onto a topological space $Y$ is 
 called a \textit{$\delta$-map} if   for every point $y\in Y,$ the preimage $f_{\delta}^{-1}(y)$ has diameter $<\delta$. 
\end{definition}
 More than half a century ago, Marde\v si\'c and Segal~\cite[Theorem 1]{MaSeg67}
 proved the following very nice characterization  result for generalized manifolds in terms of $\delta$-maps.
 
 \begin{theorem}\label{MS}
 Let $X^n$ be a compact  $n$-dimensional metric ANR such that for every $\delta>0,$
 there exists a $\delta$-map $f_{\delta}\colon X^n\to M^n_{\delta}$ of $X^{n}$ onto some (triangulated)
 oriented closed 
 topological $n$-manifold $M^{n}_{\delta}$. Then $X^{n}$ is a generalized $n$-manifold.
 \end{theorem}
 
 \begin{definition}\label{strmap}
 Marde\v si\'c and Segal called
 such a generalized $n$-manifold
  $X^{n}$ 
   \textit{manifold-like}.
 We shall call such maps
$f_{\delta}\colon X^n\to M^n_{\delta}$
 {\it structure maps}.
 \end{definition} 

\begin{remark}
Since every topological $n$-manifold (except for nonsmooth\-able 4-ma\-ni\-folds), admits a handlebody decomposition (see Quinn
{\rm{\cite{Qu82}}}), we shall hereafter neglect "triangulated".
\end{remark}

Let $d_G$ be the {\it Gromov-Hausdorff distance} which  is a complete metric on the set of
all
 isometry classes of
compact metric spaces. (Details will be given in Section \ref{s2},
 for an overview see Ferry \cite[\S 29]{Fe92}.)
In our previous paper Hegenbarth-Repov\v{s}~\cite[\S 4.3]{HeRe18}, we proved that with respect to
metric
 $d_G$,
every generalized $n$-manifold $X^{n}$
is the limit  of $2$-patch spaces, defined by
Bryant {\sl et al.} \cite{BFMW2007}.

In this paper we shall prove the following new characterization  result for manifold-like generalized $n$-manifolds - an approximation by topological $n$-manifolds  in terms of
the Gromov-Hausdorff metric $d_G$.

\begin{theorem}[Approximation Theorem]\label{A}
For every manifold-like generalized $n$-manifold $X^{n}$ and every 
$\delta>0,$ there exists a topological $n$-manifold $M^{n}_\delta$ such that $d_G(X^{n},M^{n}_\delta)<\delta$.
\end{theorem}

\begin{remark}\label{remark}
The metric on generalzed $n$-manifold $X^{n}$ is induced by a fixed embedding  $X^{n} \hookrightarrow \mathbb{R}^{m}$  of $X^{n}$ into some 
Euclidean $m$-space $\mathbb{R}^{m},$ 
for a sufficiently large dimension  $m\in \mathbb{N}$.
The metric on topological $n$-manifold  $M^{n}_{\delta}$  is then induced by an embedding 
$M^{n}_{\delta} \hookrightarrow N_{X^{n}}^{m}$ 
of $M^{n}_{\delta}$ into a small neighbourhood $N_{X^{n}}^{m} \subset \mathbb{R}^{m}$ of $X^{n}$ in $\mathbb{R}^{m}$ (see Section \ref{s2}  for more details).
\end{remark}

Edwards~\cite{Edw1978} obtained a fundamental criterion for a generalized $n$-manifold $X^{n}$ to be a topological $n$-manifold.
The first (sufficient) condition is the existence of a {\it cell-like map}
$f\colon M^n\to X^{n}$, where $M^n$ is a closed topological $n$-manifold, also called the 
 \textit{(cell-like) resolution of $X^{n}$}
 (see e.g., Mitchell and Repov\v{s} \cite{MR}).
  By the uniqueness result of Quinn
(\cite[Proposition 3.2.3]{Qu79}), any two resolutions $f_{1}\colon M_{1}^n\to X^{n}$ and $f_{2}\colon M_{2}^n\to X^{n}$
of $X^{n}$ are equivalent, i.e. for every
 $\varepsilon >0,$ 
there exists a homeomorphism 
$h_{\varepsilon} \colon M_{1}^{n} \to M_{2}^{n}$
such that $d(f_1, f_2 \circ h_{\varepsilon})<\varepsilon.$
The second (sufficient) condition is a general position type of property, the so-called {\it disjoint disks property} of $X^{n}$ (see e.g.,
Cavicchioli {\it et al.} \cite{CHR16}). 

Quinn~\cite{Qu83,Qu87} developed a
controlled surgery theory and constructed a
surgery   obstruction 
$i(X^{n})\in\mathbb Z$ to the existence of  resolutions of
generalized $n$-manifolds  $X^{n}$.
It is convenient to consider $I(X^{n}):=1+8i(X^{n})$, called the
 \textit{resolution index} (this appears naturally, passing from the quadratic 
$\mathbb L$-spectrum
to the symmetric $\mathbb L$-spectrum, see 
Ranicki \cite{Ra92}). So $I(X^{n})=1$ if and only if $X^{n}$ admits a (cell-like) resolution.

There are no known general methods for calculating Quinn's resolution index $I(X^{n})$, like there are for other invariants. 
In this paper we shall show 
that it vanishes for a certain class of manifold-like
generalized $n$-manifolds, and thus we shall prove
 that they are resolvable (see Theorem~\ref{B} below). 
First, we need some more notations (see 
Ferry \cite[\S 29]{Fe92}).

\begin{definition}
A function $\varrho\colon [0,R) \to [0,\infty)$ is called  \textit{contractible} if for every $t$,
$\varrho(t)\geq t$, and $\varrho$ is continuous at 0. Let $\mathcal{M}(\varrho,n)$ denote
the set of all compact metric spaces $M$ of dimension $\leq n$, such that for every $x \in M,$
the $r$-ball $B_r(x)=\{y\in M \mid d(x,y)\leq r\}$ contracts to $\{x\}$ inside 
the $\varrho(r)$-ball
$B_{\varrho(r)}(x)$.
\end{definition}

The following is the second main result of our paper.

 \begin{theorem}[Resolution Theorem]\label{B}
  Let 
 $X^{n}$
 be a
  gene\-ra\-lized $n$-ma\-ni\-fold
  and fix an embedding   $i \colon X^{n} \hookrightarrow \mathbb{R}^{m}$ for some $m\ge n\ge 5$.
 Let 
 $\varrho\colon [0,R) \to [0,\infty)$
  be  a contractible function and suppose that for every small 
  $\delta > 0,$
there is 
a structure map 
$f_{\delta} : X^{n} \to M^{n}_{\delta}$
  such that $M^{n}_{\delta}\in \mathcal{M}(\varrho,n)$
   with respect to the metric defined in Theorem \ref{A}.
Then $X^{n}$ is resolvable.
 \end{theorem} 
 
 \begin{remark}
We recall that the metric on generalized $n$-manifold
$X^{n}$ (resp. topological $n$-manifold $M^{n}_{\delta}$)
 is induced by the embedding
 $X^{n} \hookrightarrow \mathbb{R}^{m}$
 (resp. 
$M^{n}_{\delta} \hookrightarrow N_{X^{n}}^{m} \subset \mathbb{R}^{m}$).
\end{remark}
 
 As an application, consider the following nice result of Ferry~\cite[Proposition~29.38]{Fe92}.
 
 \begin{theorem}\label{Fe}
 Suppose that $X=\varinjlim\{M^{n}_{i}\}$, where $\{M^{n}_{i}\}\subset \mathcal{M}(\varrho,n)$,    in 
 the
 Gromov-Hausdorff metric. If $\mbox{\rm{dim}}X<\infty$, then $X$ is a generalized $n$-manifold.
  \end{theorem}
  
  It now follows by
  our   Theorem~\ref{B} that the space $X$  in  Theorem~\ref{Fe}
is in fact, a {\it resolvable} generalized $n$-manifold $X$.  
  For some related previous results on limits in the Gromov-Hausdorff metric space see
  Dranishnikov and Ferry~ \cite{DF,DF06}
   Dranishnikov {\it et al.}~\cite{DFW},
   Engel~\cite{En},
   Ferry~\cite{FePreprint,Fe94, Fe98},
   Ferry and Okun~\cite{FO},
   Grove {\it et al.}~\cite{GPW},
   Kawamura~\cite{K}, 
   and
   Moore~\cite{TEM}.  
  
  We conclude the introduction with the following very interesting open problem  related to our Theorem~\ref{B}. Recall that there are plenty of nonresolvable generalized $n$-manifolds - see e.g., Cavicchioli {\it et al.} \cite{CHR}.
  How about {\it manifold-like} generalized $n$-manifolds?
 
 \begin{question}
Does there exist, for any $n\ge 5,$ a nonresolvable manifold-like generalized $n$-manifold?
 \end{question}
  
  \section{Proof of Theorem~\ref{A}}\label{s2}
  
  Let $X^{n}$ be a manifold-like generalized $n$-manifold. 
    For any
   $\delta>0$,
  let 
 $f_{\delta}\colon X^{n}\to M^{n}_\delta$ be
 a {\it structure map} 
 from Definition \ref{strmap}. 
 We shall invoke the following result due to Eilenberg
  (see e.g. Ferry \cite[Corollary~29.10]{Fe92}).
  
   \begin{proposition}\label{L}
   For every $\delta>0,$ there exist
   a structure map
    $f_{\delta}\colon X^{n}\to M_{\delta}^{n}$ and a continuous
   map $g_{\delta}\colon M_{\delta}^{n}\to X^{n}$ such that $g_{\delta}\circ f_{\delta}\colon X^{n}\to X^{n}$ is $\delta$-homotopic to $\mbox{Id}_{X^{n}}\colon X^{n}\to X^{n}$.
    \end{proposition}  
  
This is a special case where also the following fact holds.
  
  \begin{supplement}\label{supp}
   The structure map
     $f_{\delta}\colon X^{n}\to M_{\delta}^{n}$ 
   from
   Proposition~\ref{L} is a homotopy equivalence
  with
  the
   inverse $g_{\delta}\colon M_{\delta}^{n}\to X^{n}$.
  \end{supplement}
 
  {\em Proof of Proposition~\ref{L}:} The induced map 
  \[
 ( f_{\delta})_*\colon H_*(X^{n};\mathbb Z)\to H_*(M_{\delta}^{n};\mathbb Z)
  \]
   is injective
  since $g_{\delta}\circ f_{\delta} \sim {Id}_{X^{n}}$. 
  Therefore the composition  
  \[
  H_n(X^{n};\mathbb Z)\stackrel{(f_{\delta})_{*}}{\to} 
  H_n(M_{\delta}^{n};\mathbb Z)\stackrel{(g_{\delta})_{*}}{\to}H_n(X^{n};\mathbb Z)
  \]
  is the identity,
  $(g_{\delta})_{*}\circ (f_{\delta})_{*}=(Id_{X^{n}})_{*},$
   and we have
  \[
  H_n(M_{\delta}^{n};\mathbb Z)\cong\mathbb Z, \ \ (g_{\delta})_*([M_{\delta}^{n}])=[X^{n}],
  \]
  if we choose the fundamental class 
  appropriately.  
   It follows by duality that  the induced map
   \[
   (f_{\delta})_*\colon H_*(X^{n};\mathbb Z)\to H_*(M_{\delta}^{n};\mathbb Z)
   \]
   is also surjective and
   that
     $f_{\delta}\colon X^{n}\to M_{\delta}^{n}$
     and $g_{\delta}\colon M_{\delta}^{n}\to X^{n}$
      are both of degree 1.
In particular, since the map $f_{\delta}\colon X^{n}\to M_{\delta}^{n}$ is of degree 1, it now
follows that the induced map 
\[
(f_{\delta})_*\colon \pi_1(X^{n})\to\pi_1(M_{\delta}^{n})
\]
 is surjective (see Browder~\cite[Proposition~1.2]{Browd72}).  
Since 
$(f_{\delta})_*\colon \pi_1(X^{n})\to\pi_1(M_{\delta}^{n})$ is
also  injective, it is in fact, an isomorphism. 

Now, arguing as above, we can show that 
$f_{\delta}\colon X^{n}\to M_{\delta}^{n}$ 
 induces  isomorphisms  in homology with coefficients in
group rings. It therefore follows by Ferry \cite[Theorem 7.4]{FePreprint} that
 $f_{\delta}\colon X^{n}\to M_{\delta}^{n}$ 
  is indeed a homotopy equivalence
  with the inverse
$g_{\delta}\colon M_{\delta}^{n}\to X^{n}$. 
 This completes the proof of Proposition~\ref{L}.
\hfill $\Box$

\begin{definition} The {\it Gromov-Hausdorff distance} between any compact metric spaces $X$ and $Y$ is defined as follows:
For  any closed subsets $X$ and $Y$ of a compact metric space $(Z,d),$ 
and any
$\delta >0,$
define their neighborhoods  
\[
N_{\delta}(X):=\{z\in Z \mid d(z,X)<\delta\},
\]
 and  
\[
N_{\delta}(Y):=\{z\in Z \mid d(z,Y)<\delta\}
\]
and define the following distances
\[
d_Z(X,Y):=\inf\{\delta>0 \mid X\subset N_{\delta}(Y)  \ and \  Y\subset N_{\delta}(X)  \}
\]
and 
\[
d_G(X,Y):=\inf\{d_Z(X,Y)\mid X,Y \;are \ isometrically \ embedded  \ in \ \, Z\},
\]
where $Z$ ranges over all compact metric spaces.
\end{definition}
\begin{remark}
The Gromov-Hausdorff convergence is a notion of
 convergence of metric spaces which is a generalization of the 
classical
Hausdorff convergence.
The Gromov-Hausdorff distance was introduced  in 1975 
by 
Edwards~\cite{1}
and 
then rediscovered and generalized in 1981 by 
Gromov~\cite{3}
 (see also
Tuzhilin~\cite{2}).
\end{remark}
To determine $d_G(X^{n},M_{\delta}^{n})$ for 
a structure map
 $f_{\delta}\colon X^{n}\to M_{\delta}^{n}$, the choice of the metric is important. We choose an embedding 
$X^{n}\hookrightarrow \mathbb{R}^m$, and take on $X^{n}$ the metric
induced from $\mathbb{R}^m$. It is important to note that the property of $f_{\delta}\colon X^{n}\to M_{\delta}^{n}$ 
being
 a structure map
  does not depend on the choice of the metric on $M_{\delta}^{n}$.
It will be appropriately chosen below.

Let $f_{\delta}\colon X^{n}\to M_{\delta}^{n}$ be
 a structure map 
 with the
inverse $g_{\delta}\colon M_{\delta}^{n}\to X^{n},$ such that $g_{\delta}\circ f_{\delta}$ is
$\delta$-homotopic to ${Id}_{X^{n}}$
for a given small $\delta>0$  (see Proposition~\ref{L}).
In the sequel, let 
\[
i\colon X^{n}\hookrightarrow N_{\delta}:= N_{\delta}(X^{n}\hookrightarrow{\mathbb R}^m)
\]
denote the inclusion of $X^{n}$ into a $\delta$-neighbourhood
$ N_{\delta}$
 of $X^{n}$ in $\mathbb{R}^m$.

Since by hypothesis, $X^n$ is  manifold-like, it follows that for arbitrary small $\delta'>0$, there exists an embedding $j\colon M_{\delta}^{n}\hookrightarrow N_{\delta}$ with
$d(i\circ g_{\delta}, j)<\delta'$ (see Rourke and Sanderson~\cite[General Position Theorem for Maps 5.4]{RourSand72}).
 These maps can be represented by the following diagram

\begin{equation}\label{d1}
	\begin{tikzpicture}[baseline=-1cm, node distance=2cm, auto]
	  \node (LU) {};
	  \node (XX) [node distance=1cm, below of=LU] { };
	  \node (NU') [right of=LU] {$X^{n}$};
	  \node (NU) [below of=NU'] {$M_{\delta}^{n}$};
	  \node (RU) [right of=NU'] {};
	  \node (X2) [node distance=1cm, below of=RU] {$N_{\delta}$};	 
	  \draw[->, font=\small] (NU) to [bend left] node {$g_{\delta}$} (NU');
	  \draw[->, font=\small] (NU') to [bend left]  node {$f_{\delta}$} (NU);
	  \draw[->, font=\small] (NU') to node {$i$} (X2);	 
	  \draw[->, font=\small] (NU) to node [swap] {$j$} (X2); 
	\end{tikzpicture}
	\end{equation}	

We choose on $M_{\delta}^{n}$ the metric induced on $j(M_{\delta}^{n})\subset \mathbb{R}^m$.
Since 
\[
d(i\circ g_{\delta}, j)<\delta',
\]
 we
can
 deduce the following
 \[
 d(i(x), j(M_{\delta}^{n}))\leq d(i(x), (i\circ g_{\delta}\circ f_{\delta})(x))+d((i\circ g_{\delta}\circ f_{\delta})(x), j(M_{\delta}^{n}))
<\delta+\delta',
\]
 i.e.,  
\[
i(X^{n})\subset N_{\delta+\delta'}(j(M_{\delta}^{n})\subset\mathbb{R}^m)
\]
 (see also Remark~\ref{R} below).

Of course, $N_{\delta}$ and $N_{\delta+\delta'}(j(M_{\delta}^{n})\subset\mathbb{R}^m)$
belong to a compact subset $Z$ of $\mathbb{R}^m$ with the induced metric. We obtain the following
\[
d_G(X^{n},M_{\delta}^{n})\leq d_Z(X^{n},M_{\delta}^{n})<\delta+\delta'.
\]
 Now $\delta$ and $\delta'$ can be
chosen to be arbitrarily small, thus we have completed the proof of Theorem~\ref{A}. \hfill $\Box$

\begin{remark}\label{R}
Recall that 
\[
d(z,A):=\inf \{d(z,a)\mid a\in A\},
\]
 where $A\subset Z$ is a compact subset of the metric space $Z$.
For $z,z'\in Z$, the inequality 
\[
d(z',a)\leq d(z,z')+d(z,a)
\]
 implies the inequality 
 \[
 d(z',A)\leq d(z',z)+d(z,A),
 \]
 which was used above.
\end{remark}

\section{Proof of Theorem~\ref{B}}\label{s3}

In this section, we shall apply the controlled surgery sequence to prove Theorem~\ref{B}. For more details on this important subject we refer to
 Bryant {\it et al.}~\cite{BFMW96}, 
 Cavicchioli {\it el at.}~\cite{CHR16}, 
 Ferry~\cite{Fe10, FPPreprint,FP}, 
 Mio~\cite{Mio2000},
Pedersen {\it et al.}~\cite{ PQR03},
Pedersen and Yamasaki~\cite{ PY},
Quinn~\cite{ Qu83,Qu87},
Ranicki and Yamasaki~\cite{ RY},
and
Yamasaki~\cite{ Y}.
 
Let $\mathbb{L}$ denote the periodic $\mathbb{L}$-spectrum, i. e. $\mathbb{L}_0=
\mathbb{Z}\times G/{\mbox{\em TOP}}$, and $\mathbb{L}^+$ is its connected covering spectrum
with $\mathbb{L}_0^+=G/{\mbox{\em TOP}}$. 
Now, if ${\mathcal{S}}_{\varepsilon}\left(
\begin{array}{c}
X^{n}\\
\downarrow{{Id}}\\
X^{n}
\end{array}
\right)\neq \emptyset,$
then there exists an exact sequence 
\[
\dots
\rightarrow
H_{n+1}(X^{n};\mathbb{L}^+)
\rightarrow
H_{n+1}(X^{n};\mathbb{L})
\rightarrow
{\mathcal{S}}_{\varepsilon}\left(
\begin{array}{c}
X^{n}\\
\downarrow{Id}\\
X^{n}
\end{array}
\right)
\rightarrow
H_{n}(X^{n};\mathbb{L}^+)
\rightarrow
\dots
\]   
\noindent
Elements of ${\mathcal{S}}_{\varepsilon}\left(
\begin{array}{c}
X^{n}\\
\downarrow{Id}\\
X^{n}
\end{array}
\right)$ are equivalence classes of $\varepsilon$-homotopy equivalences 
$M^n\stackrel{h}{\to}X^{n}$
(measured in $X^{n}$), with $M^n$ a closed (oriented) topological $n$-manifold.
\begin{definition}
Two elements 
\[
M^{n}_1\stackrel{h_1}{\to}X^{n}, M^{n}_2\stackrel{h_2}{\to}X^{n}\in 
 {\mathcal{S}}_{\varepsilon}\left(
\begin{array}{c}
X^{n}\\
\downarrow{{Id}}\\
X^{n}
\end{array}
\right)
\]
 are said to be $\varepsilon$-{\it related}
if there exists a homeomorphism $\varphi\colon M^{n}_1\to M^{n}_2$ such that $h_2\circ \varphi$ is
$\varepsilon$-homotopic to $h_1$. 
\end{definition}
 \begin{remark}
 Being $\varepsilon$-related does not define an equivalence
 relation, but it is a part of the following assertion: There exists an $\varepsilon_0 >0$ depending only on 
 $X^{n},$ such that for every $\varepsilon\leq \varepsilon_0,$
 this
  becomes an equivalence relation. 
  \end{remark}
  
  For $p+q=n+1,$
 it follows from the spectral sequences 
 \[
 E^2_{pq}=H_p(X^{n}; \pi_q(\mathbb{L}))\Rightarrow
 H_{p+q}(X^{n}; \mathbb{L})
 \]
  and
 \[
 E^{+2}_{pq}=H_p(X^{n}; \pi_q(\mathbb{L}^+))\Rightarrow H_{p+q}(X^{n}; \mathbb{L}^+)
 \] 
  that 
  \[
  E^{+2}_{pq}=E^2_{pq},
  \]
   hence 
 \[
 H_{n+1}(X^{n}; \mathbb{L}^+)\cong H_{n+1}(X^{n}; \mathbb{L}).\]
  Moreover,
 $H_{n}(X^{n}; \mathbb{L}^+)\to H_{n}(X^{n}; \mathbb{L})$ must be injective. It follows that if  
\[
{\mathcal{S}}_{\varepsilon}\left(
\begin{array}{c}
X^{n}\\
\downarrow{{Id}}\\
X^{n}
\end{array}
\right)\neq \emptyset.
\]
then  it consists of only one element
\[
\hbox{card}\left[
{\mathcal{S}}_{\varepsilon}\left(
\begin{array}{c}
X^{n}\\
\downarrow{{Id}}
\\
X^{n}
\end{array}\right)\right]=1.
\]

\begin{proposition}\label{p}
Let $X^{n}$ be a generalized $n$-manifold. Then $I(X^{n})=1$ if and only if
\[
{\mathcal{S}}_{\varepsilon}\left(
\begin{array}{c}
X^{n}\\
\downarrow{{Id}}\\
X^{n}
\end{array}
\right)\neq \emptyset,
\]
 i.e. for every $\varepsilon\leq\varepsilon_0,$ there exists an
$\varepsilon$-homotopy equivalence $M^n\stackrel{h}{\to} X^{n}$.
\end{proposition}

\begin{proof} The proof  is standard, see
e.g.,
Mio \cite[\S 3]{Mio2000} or
Bryant {\it et al.}  \cite[p. 444]{BFMW96}.
\end{proof} 

In order to prove Theorem \ref{B}, 
we have to show that  for
each $\varepsilon\leq\varepsilon_0,$
there exists
for 
every
${\mathcal M}(\varrho, n)$-like generalized manifold
$X^{n},$
 an $\varepsilon$-homotopy equivalence
$h_{\varepsilon}\colon M^n\to X^{n}$. This follows from Theorem~\ref{A} and  Ferry~\cite[Theorem~29.20]{Fe92}.

\begin{theorem}\label{F}
Let $\varrho\colon [0,R)\to[0, \infty)$ be a contractible function and let $Y$ and $ Z$ be any compact metric
 spaces. Then for every $\varepsilon>0,$
 there  exists $\delta>0$ such that if 
 $Y,Z\in {\mathcal M}(\varrho, n)$ and $d_G(Y,Z)<\delta,$ then $Y$ and $Z$ are 
 $\varepsilon$-homotopy equivalent.
 Here, $\delta=\delta(\varepsilon,\varrho)$ depends  on $\varepsilon$ and $\varrho$, but not  on $Y,Z$.
\end{theorem}

Let us provide some more details: We equip generalized
 $n$-manifold 
  $X^{n}$
 with the metric given by an embedding
  $X^{n} \hookrightarrow \mathbb{R}^{m}$ 
   of 
   $X^{n}$
    into some 
    $\mathbb{R}^{m},$ 
for a sufficiently large 
$m \in \mathbb{N}$, 
see Theorem \ref{A} and Remark \ref{remark}.
 By Ferry~\cite[Theorem 29.14]{Fe92},
  $X^{n}$  with this metric belongs to 
   $\mathcal{M}(\varrho, n)$ 
   for some contractible function $\varrho\colon [0,R) \to [0, \infty)$.

By hypothesis, we can now choose a sequence $\{\varepsilon_i >0 \}_{i\in \mathbb{N}}$ such that
\[
\lim_{i \to +\infty} \varepsilon_i =  0, \quad  \sum_{i=1}^{\infty}\varepsilon_i < \infty,
\]
 and then invoking  Theorem~\ref{F},
 obtain a sequence 
 \[
 \{\delta_i:=\delta_i(\varepsilon_{i},\varrho)>0\}_{i\in \mathbb{N}}.
 \]

By  Theorem~\ref{A}, then there exists a sequence of 
 closed topological $n$-manifolds
$\{M^n_{\delta_{i}}\}_{i\in \mathbb{N}}\subset {\mathcal M}(\varrho, n),$
with respect to the metric obtained by 
embedding  
$M^n_{\delta_{i}} \hookrightarrow N_{X^{n}}^{m} \subset \mathbb{R}^{m}$ 
 each $M^n_{\delta_{i}}$ into  a small neighbourhood $N_{X^{n}}^{m}$ of generalized $n$-manifold $X^{n}$ in $\mathbb{R}^{m},$
such that
\[
 d_G(M^n_{\delta_{i}},X^{n})<\delta_i,
 \
 \hbox{
 for every
 }
 \
 i\in \mathbb{N}.
 \]
 
Therefore every topological $n$-manifold $M^n_{\delta_{i}}$ is
 $\varepsilon_i$-homotopy equivalent to $X^{n}$.
This proves Theorem~\ref{B}. \hfill $\Box$

\section*{Acknowledgements}
This research was supported by the Slovenian Research Agency grants P1-0292, J1-4031, J1-4001, N1-0278, N1-0114, and N1-0083. 
 We thank the referee for comments and suggestions.

\end{document}